\documentclass[11pt,
twoside]{amsart}
\usepackage{lmodern} 
\usepackage{soul, wrapfig}
\usepackage{graphicx}
\usepackage{color}
\usepackage{amsmath}
\usepackage{amssymb}
\usepackage{amsfonts}
\usepackage{latexsym, amsthm, mathrsfs}
\usepackage{hyperref, breakurl}
\usepackage[OT1]{fontenc}
\usepackage{fancybox}
\usepackage{amscd}
\usepackage{enumitem}
\usepackage{fontenc}
\usepackage{amsthm}
\usepackage{graphicx}
\usepackage[english]{babel}
\usepackage{tikz-qtree}
\usepackage{rotating}
\usepackage{stmaryrd}

\numberwithin{equation}{section}
\newtheorem{theorem}{Theorem}[section]
\newtheorem{lemma}[theorem]{Lemma}

\theoremstyle{plain} 

\newtheorem{proposition}[theorem]{Proposition}
\newtheorem{corollary}[theorem]{Corollary}

\newtheorem{claim}[theorem]{Claim}

\newtheorem*{maintheorem*}{Main Theorem}
\newtheorem*{conjecture*}{Conjecture}
\newtheorem*{theorem*}{Theorem}
\newtheorem*{proposition*}{Proposition}
\newtheorem*{corollary*}{Corollary}

\theoremstyle{definition} 
\newtheorem{definition}[theorem]{Definition}

\theoremstyle{remark}  
\newtheorem{remark}[theorem]{Remark}
\newtheorem{example}[theorem]{Example}
\newtheorem*{remarks*}{Remarks}
\newtheorem*{remark*}{Remark}
\newtheorem*{claim*}{Claim}

\renewcommand{\phi}{\varphi}

\newcommand{\cohen}{\poset{C}}
\newcommand{\bC}{\poset{C}}
\newcommand{\conc}{\smallfrown}

\newcommand{\club}{\mathrm{Club}}

\newcommand{\clubsilver}{\silver^\club}

\newcommand{\DDelta}{\mathbf{\Delta}}

\newcommand{\even}{\text{Even}}
\newcommand{\ideal}{\mathcal}

\newcommand{\ifif}{\Leftrightarrow}
\newcommand{\infinite}{\infty}

\newcommand{\Finite}{\mathsf{Fin}}

\newcommand{\force}{\Vdash}
\newcommand{\filter}{\mathbf}

\newcommand{\ns}{\textsf{NS}}
\newcommand{\nb}{\mathbf{NB}}
\newcommand{\nuv}{\mathbf{NUV}}
\newcommand{\nv}{\mathbf{NV}}

\newcommand{\odd}{\text{Odd}}

\newcommand{\poset}{\mathbb}

\newcommand{\restric}{{\upharpoonright}}

\newcommand{\silver}{\poset{V}}

\newcommand{\SSigma}{\mathbf{\Sigma}}

\newcommand{\term}{\text{term}}

\newcommand{\1}{\mathbf{1}}
\newcommand{\0}{\mathbf{0}}

\newcommand{\nc}{\newcommand}
\nc{\marginparr}[1]
{\marginpar{\makebox[4mm]{} {#1}}}

\nc{\bP}{\mathbb P}
\nc{\nothing}[1]{}
\nc{\comment}[1]{#1}

\nc{\nco}{\DeclareMathOperator}
\nco{\codes}{code}
\nco{\taille}{taille}
\nco{\ter}{ter}
\nco{\rk}{rk}
\nco{\llower}{lower}
\nco{\order}{o}
\nco{\Nm}{Nm}
\nco{\ppower}{pp}
\nco{\pcf}{pcf} 
\nco{\tcf}{tcf} 
\nco{\tlim}{tlim} 
\nco{\limtext}{lim} 
\nco{\prodt}{{\textstyle \prod}}
\nco{\symdiff}{\triangle}
\nco{\dom}{dom}
\nco{\card}{card}
\nco{\lh}{lh}
\nco{\lt}{lt}
\nco{\lgg}{lg}
\nco{\hgt}{ht}
\nco{\rge}{range}
\nco{\otp}{otp} 
\nco{\trunk}{tr}
\nco{\nex}{next}
\nco{\reduction}{red}
\nco{\supt}{supt}
\nco{\supp}{supp}
\nco{\Lim}{Lim}
\nco{\Leb}{Leb}
\nco{\modd}{mod}
\nco{\invariant}{inv}
\nco{\id}{id}
\nco{\RO}{RO}
\nco{\poss}{pos}
\nco{\Inc}{Inc} 
\nco{\Fn}{Fn}
\nco{\add}{add}
\nco{\borel}{Bor}
\nco{\cof}{cof}
\nco{\cov}{cov}
\nco{\height}{ht}
\nco{\lev}{Lev}
\nco{\levy}{Coll}
\nco{\non}{non}
\nco{\ot}{ot}
\nco{\rank}{Rank}
\nco{\splitting}{Split}
\nco{\splitlevel}{ns}
\nco{\stem}{stem}
\nco{\successor}{succ}
\nco{\Succ}{Suc}
\nco{\splsuc}{splsuc}
\nco{\Lev}{Lev}
\nco{\N}{\omega}
\nco{\uppersilver}{\poset{V}_\delta}
\nco{\cantor}{2^{\N}}
\nco{\uv}{\uppersilver}
\nco{\Power}{\mathcal{P}}

\nco{\nl}{\mathbf{NL}}

\nc{\la}{\langle}
\nc{\ra}{\rangle}

\begin{document}

\title[The non-democratic side of regular sets
]{The non-democratic side of regular sets}
\author{Giorgio Laguzzi}
\thanks{Albet-Ludwig-Universit\"at Freiburg, Mathematisches Institute, Ernst-Zermelo str. 1, 79104
Freiburg im Breisgau, Germany; Email: giorgio.laguzzi@libero.it}

\begin{abstract}
We analyse the role and the possible interpretations of regular sets studied in descriptive set theory and forcing theory from the point of view of social choice theory and we provide some insights about their \emph{non-democratic nature}. Moreover we also show that equity and Pareto principles can rule out the non-democratic aspects.
\end{abstract}
\maketitle

\section{Introduction}
The general question addressed by social choice theory is to extrapolate a reasonable social choice out of a set of individual choices. Arrow's original result (\cite{Arrow}) proved the impossibility of a social choice function working on three alternatives and satisfying Unanimity (U), Independence of Irrelevant Alternatives (IIA) and Non-Dictatorship (ND); then in \cite{Fish69} Fishburn turned Arrow's theorem into a \emph{possibility} result, showing that when dealing with an infinite set of individuals the existence of such a social choice function is somehow linked to the existence of non-principal ultrafilters, thus exhibiting a connection between social choice functions satisfying a non-dictatorial behavior and non-constructive objects. Many variants and more results have been extensively developed through the years by many other authors following this line of research.
For example, Lauwers and Van Liedekerke introduced a more systematic model-theoretic approach in \cite{Lauwers2}; Litak provided a game-theoretic approach with infinite games and the axiom of determinacy in \cite{Litak}; Mihara investigated the structure in terms of oracle computability and Turing degrees of the social choice function and the invisible dictator of Fishburn's theorem in \cite{Mihara1}.

This paper proposes more insights from a slightly different perspective about such a connection involving other irregular sets from set theory and topology, and providing some interpretations in terms of social choice functions. The use of irregular sets like non-Lebesgue, non-Ramsey sets have been used also in connection with social welfare relations satisfying Paretian principles, anonymity and other equity principles (e.g., Lauwers \cite{Lauwers1} and Zame \cite{Zame}). 

Focusing on the part concerning social choice theory on countably infinite population $P$, (with also a brief digression on the uncountable case in Section 7), we want to study some non-constructive objects from topology and descriptive set theory, and show that the dependence faced by Fishburn of more democratic social choice functions from non-constructive objects seems to be not only due to the starting principles used by Arrow (U and IIA), but probably hides a more structural motivation, and indeed even shows up in case of two alternatives.

Since the paper somehow attempts to blend together a mostly social choice-theoretic side and a mostly set-theoretic side, one needs to organize the two parts accordingly. Section \ref{s2} is designed to discuss some motivations in economic and social choice theory regarding the role of infinite populations, and so it can be seen mostly as a philosophical part. Section \ref{s3} starts to introduce the more technical tools from set theory, focusing on the  interpretation in the framework of social choice theory, and thus it may be understood as a bridge between the speculative part and the technical one. From Section \ref{s4} on, we deeply go into the technical and set theoretical issues, providing in the end the corollaries/interpretations in social choice theory with a focus on the case with two alternatives. Section \ref{s5} is concerning the case with more than two alternatives. In Section \ref{s6} we introduce equity and Pareto principles for social welfare relations and prove that these principles could be used to prevent anti-democratic social choice functions. Section \ref{s7} is a brief digression for uncountable populations. Finally, Section \ref{s8} summarizes the concluding remarks about the use of the axiom of choice and the axiom of determinacy in the context of social choice and economic theory. 

\section{Some considerations on infinite populations}\label{s2}

The study and the analysis of models with infinitely many individuals has a long-standing tradition in social choice and economic theory. Moreover in such a context the use of set-theoretic axioms like the axiom of choice (AC) and the axiom of determinacy (AD) has played a role. In this section, we attempt to give some overview and a contribution about the discussion on the role of  infinite populations and the use of these axioms from mathematical logic. The overview is far from being exhaustive and it is just mainly centered on what is pertaining mostly our further discussion. My position will be moderately in favour of a use of some fragment of AC, and rather positively oriented in analysing infinite populations. 

In \cite{Litak} Litak deals with an interesting and profound analysis and some criticisms about some use of infinite populations and non-constructive objects in economic and social choice models. In particular, he suggests to take in consideration as a method of selection of the legitimate use of AC (and the suitable use of the infinite setting in general) the so-called \emph{Hildenbrand criterion}, asserting: ``\emph{The relevance of the ideal case to the finite case has to be established}". Although the paper remarkably contains a lot of stimulating insights and I agree with some positions expressed in that paper, and in general with some criticisms, I would argue for a more optimistic position about the use of infinite sets of individuals, and try to be more nuanced about the preference between AC and AD in this context. In this section I mainly focus on the use of infinite population of individuals, and refer to the last section on the concluding remarks for the discussion about AC and AD.

In favor of an appropriate use of economic models with infinitely many individuals is Aumann in his pioneering  \cite{Au64}, where the author started a systematic study of markets in perfect competition by using a continuum of traders. The idea behind using infinitely many traders adopted by Aumann is that in analysing markets with perfect competition one needs a model/framework where the influence of each individual participant in ``negligible''. The author draws a parallel with the use of a continuum of particles in fluid mechanics to support the idea that, even though idealized, it can serve as a fruitful setting to understand the dynamic behavior of fluids, in some cases more than by considering the whole fluid as composed by $n$-many particles and studying the distinct movements of each particle. Beyond that, a very significant passage is the following \cite[pg. 41]{Au64}:

\vspace{2mm}
``\emph{It should be emphasized that our consideration of a continuum of traders is not merely a mathematical exercise; it is an expression of an economic idea. This is underscored by the fact that the chief result holds only for a continuum of traders - it is false for a finite number.}" 
\vspace{2mm}

This passage is rather important as it seems to suggest that the use of an infinite amount of individuals permits to analyse a certain idealized economic framework even better than the more realistic finite case, and moreover it permits to state that this choice is not just \emph{``merely a mathematical exercise; it is an expression of an economic idea"}. Even if not stated explicitly, it seems to be clear from that passage and from the context, that Aumann is strongly supporting the idea of using infinitely many traders as the \emph{most suitable} model for perfect competition. We remark that the purpose here is neither to support any preference between market economy in perfect competition vs monopolistic economy, nor other issues regarding political economy (which is far from the topic of this paper), but merely to stress that Aumann suggests a choice of an infinite population based on a reasoning pertaining models for economic theory, and not just as a mathematical abstraction. In \cite{Au64} while Aumann seems to therefore be rather positively oriented in the use of infinitely many traders, he seems to be suspicious about an unrestricted use of AC, as stated on page 44: ``\emph{Non-measurable sets are extremely ``pathological"; it is unlike that they would occur in the context of an economic model}".

In \cite{Mihara1} and \cite{Mihara2} Mihara draws a connection between computability and social choice theory. The author expresses a rather positive approach in using infinite populations. Specifically the author introduces also an interpretation and support some ideas that we find very suggestive. Beyond the usual interpretation of an infinite population as consisting of infinitely many people extending indefinitely through an infinite horizon, Mihara's other idea relies on the observation that an infinite set of individuals need not be understood as an infinite set of physical persons only (which might make the idealized setting rather abstract and unrealistic), but could perfectly be seen in a framework where there is a finite set of different persons but there is uncertainty described by a countably infinite set of states.  As an explicit example we mention the following passage \cite[pg. 4]{Mihara2}:

\vspace{2mm}
``\emph{This derivation of an infinite ``society'' as well as the domain restrictions might
seem artificial. However, they are in fact natural and even have some advantages.
First, inter-state comparisons are avoided, in the same sense that inter-personal
comparisons are avoided in Arrow's setting. Second, in this formulation, people
can express their preferences without estimating probabilities.}" 
\vspace{2mm}

 Mihara provides also an explicit example to support this interpretation (\cite[section 2.1.3]{Mihara2}). As a contribution to support Mihara's interpretation I also suggest the following one. 

Consider the case of a single person who has to make a choice between two alternatives (i.e., two candidates $a$ and $b$) and that this decision must be explicitely make at a certain point in the coming future; meanwhile her decision might ``oscillates'' depending on the particular instant we consider before the voting time. More formally, let $t_s$ denote the instant where we start to collect the preferences of the given person and $t_f$ be the final instant when the decision will be made. If we interpret the time as a continuum we could understood the interval $[t_s,t_f)$ as an infinite set (we can even reduce the focus on the rational instants, so to have a countable dense set of instants in $[t_s,t_f)$). In such a way we can see the final choice that the person will make at the voting instant $t_f$ as a choice \emph{aggregating} all individual choices made in all preceding instants; for example, if the person has thought of giving a preference to the candidate $a$ instead of $b$ in all instants, i.e. the sequence of preference is of the form $(a,a,a,a,a, \dots)$, at instant $t_f$ we expect an aggregating choice function giving $a$ as the final choice (unanimity principle); on the contrary if the person has denoted uncertainty, i.e., the sequence of individuals choices is of the form $(a,a,b,b,a,b,a,a,b,a,b,\dots)$ then the final aggregating choice will not be so clear. Hence under this interpretation the final choice at instant $t_f$ can be analysed as an aggregating choice of all rational instants in $[t_s,t_f)$. Note that the ordered sequence does not follow the order in terms of \emph{time}, but this should not be considered \emph{a priori} so bounding, since  an aggregating choice function need not be influenced on how the set of individuals is ordered, and in case of a countable set the well-order does not require any use of AC. Also note that this simple example could be seen as a particular case of Mihara's setting, since the individual choices at all rational instants $t \in [t_s, t_f)$ can be viewed as the states of uncertainty described by Mihara. A possible objection might be that in this example we treat $t_f$ as a kind of privileged instant, and not as the other $t$'s. However this is justified by the fact that the specific \emph{final} instant $t_f$ can be seen as the acme which aggregates all states of uncertainty preceding that specific moment $t_f$.

The reader interested in more details about the use of infinite populations both in social choice theory and social welfare pre-orders, also in relation with non-constructive objects could see the following selected list of papers: \cite{Chi2}, \cite{Fish69}, \cite{Lauwers2} (about social choice functions), and \cite{Chi96}, \cite{Lauwers1}, \cite{Zame}, \cite{Asheim2010} (about social welfare relations).

\section{Examples, basic notions and irrelevant coalitions} \label{s3}

Let us fix the general setting we aim to investigate.
We focus on the case of an alternative between two candidates $0,1$ and of a countably infinite population $P=\N$. We will focus on the case with more alternatives (and even countably infinite many) in Section \ref{s5}. 
Given any individual $i \in \N$ we denote by $x(i) \in \{ 0,1 \}$ the choice that the individual $i$ makes between the two candidates. The sequence $x \in 2^{\N}$ then represents all individual choices. In this formalization, given a set $F \subseteq 2^{\N}$ we can see it as a choice function assigning a value either $0$ or $1$ to a given sequence of individual preferences $x \in 2^{\N}$ as follows
\[
F(x)=1 \Leftrightarrow x \in F.
\]
Throughout the paper we will identify $F \subseteq 2^{\N}$ and its corresponding social choice function on $2^{\N}$.

Given a partial function $f: \N \rightarrow \{ 0,1 \}$ we use the notation $$N_f:= \{x \in 2^{\N}: \forall n \in \dom(f) (x(n)=f(n))  \}.$$

Moreover given a set $X$ and a subset $A \subseteq X$ we use the notation $A^c:= X \setminus A$, i.e. the complement of $A$. 

We start with informally presenting an example that we then turn into a more precise treatment in the second part of this section, and with further results in the subsequent sections.

\begin{definition}
We define $N_f$ to be a \emph{Silver condition} iff the partial function associated $f: \dom(f) \rightarrow \{0,1\}$ is such that the complement of the domain of $f$, in symbol $\dom(f)^c$, is infinite. 
We denote the set of all Silver conditions by $\silver$.
\end{definition}

It is well-known (and easy to check) that the $N_f$'s form a basis for a topology, which is called \emph{Silver topology} or \emph{Doughnuts topology}. 

\begin{definition}
A set $A \subseteq 2^{\N}$ is said to satisfy the \emph{Silver property} (or $\silver$-property) iff 
\[
\exists N_f \in \silver (N_f \subseteq A \vee N_f \cap A = \emptyset).
\]
\end{definition}
We refer to these sets also with the term \emph{Silver sets} (or \emph{$\silver$-sets}).
A well-known and easy argument shows that any set satisfying the Baire property w.r.t. the Silver topology is also a Silver set. 
A standard application of AC gives a non-Silver set. On the other side, it is possible to construct by using the method of forcing a model of set theory without AC where all sets are Silver (see \cite{BLH2005}). Non-Silver sets are usually considered \emph{irregular} sets from the point of view of descriptive set theory and so somehow they are considered on the negative side. In what follows we give a first glance that, on the contrary, from the point of view of social choice theory, non-Silver sets seems to share a more positive side.

\emph{Interpretation in the social choice framework.} Given a Silver condition $N_f$ put $a:= \dom(f)^c$. 
On the notation-side, in the context of social choice theory, we often refer to a set of individuals as a \emph{coalition}.  
Let us interpret $a$ as a coalition consisting of individuals who did not make a choice yet, and $f: a^c \rightarrow \{  0,1 \}$ as the function assigning to any individual $i \in a^c=\dom(f)$ the corresponding choice $f(i)$. 

Now assume we pick a Silver set $F$, and so we can find a Silver condition $N_f$ such that $N_f \subseteq F$ or $N_f \cap F = \emptyset$. Equivalently we can say that
\begin{itemize}
\item $\forall x \in N_f$, $F(x)=0$, or
\item $\forall x \in N_f$ $F(x)=1$.  
\end{itemize}
 So assume for instance that the former case occurs, i.e. $\forall x \in N_f$, $F(x)=0$. By definition of a Silver condition, every sequence of individual choices $x \in N_f$ has the property that every individual $i \in a^c$ chooses $f(i)$, i.e. $x(i)=f(i)$. That means  that no matter what the individuals $j \in a$ will choose, the social choice function $F$ assigns the value $0$ to all sequences of individual choices in $N_f$. In other words, the latter statement asserts that the single choices of the infinitely many individuals of the coalition $a$ are irrelevant to determine the social choice assigned. An equivalent argument of course holds in case $\forall x \in N_f$ $F(x)=1$.

Note that the argument above reveals a potentially non-democratic behavior of the social choice function $F$ whenever such $F$ is a Silver set, as it ignores the choices of infinitely many individuals of the population. Or, put it differently, if $F$ is a social choice function which is a Silver set, then we can find a co-infinite coalition that is already able to determine the social choice, no matter what the other infinitely many individuals not in this coalition will choose. 
On the contrary, if a social choice function is non-Silver, that means that for every Silver condition $N_f$ we can always find a sequence of individual choices $x_0 \in N_f$ whose associated social choice is $0$, i.e., $F(x_0)=0$, and another sequence $x_1 \in N_f$ whose social choice is $1$, i.e. $F(x_1)=1$. That in particular means that the single choices of the individuals $j \in a$ play a relevant role in the final decision. 
Hence, under this point of view non-Silver sets looks ``more democratic" than Silver sets, as they seem to pay more attention on the single choices of each individual. 

\begin{remark} \label{remark2}
One might object that the coalition $a$ of irrelevant individuals may be very slim compared to its complement, and so that there is nothing anti-democratic in that. For instance if 
\[
a:= \{ 1, 10, 100, 1000, \dots \}:= \{ 10^n: n \in \N \}
\]
then the set of irrelevant individuals is very \emph{sparse} compared to the set of relevant individuals $a^c$, and so it is plausible that the social choice function $F$ makes a decision which is not affected by individuals in $a$. 
\end{remark}

To overtake Remark \ref{remark2}, we need to consider Silver conditions having \emph{non-small} set $a$, so that the associated function $f$ cannot consider \emph{irrelevant} any \emph{non-small} set of  individuals. 

Before going to this study, we need to turn the intuitive observations above into precise and formal definitions and notions.

\begin{definition} \label{irrelevant}

Given a social choice function $F: 2^{\N} \rightarrow \{  0,1 \}$, $b \subseteq \N$ and $f : \N \rightarrow \{ 0,1 \}$ partial function with $\dom(f)=b^c$, we say that $b \subseteq \N$ is $(F,f)$-\emph{irrelevant} iff for every $y,z \in N_f:=\{ x \in 2^{\N}: \forall n \in \dom(f) (f(n)=x(n)) \}$, one has $F(z)=F(y)$ (or in other words $N_f \subseteq F$ or $N_f \cap F = \emptyset$). 

We say that $b \subseteq \N$ is \emph{$F$-irrelevant} iff there exists a partial function $f: \N \rightarrow \{0,1\}$ with $\dom(f)=b^c$ such that $b$ is $(F,f)$-irrelevant.
\end{definition}

Definition \ref{irrelevant} formalizes the idea that given $b \subseteq \N$ being $F$-irrelevant means that there is an $f$ deciding the individual choices of all other members not in $b$ which makes the choice of all members in $b$ irrelevant for the final collective choice $F$. 

In the following definitions, $I \subseteq \Power(\N)$ should be understood as a family of \emph{small} sets (e.g., in many cases $I$ is an ideal). Moreover we use the notation
\[
\begin{split}
I^+:=&  \{  b \subseteq \N: b \notin I \} \\
I^*:=&  \{  b \subseteq \N: b^c \in I \} \\
\end{split}
\]
\begin{definition} \label{democratic}
Let $F$ be a social choice function.

We say that $F$ is \emph{anti-democratic w.r.t. $I^+$} (or simply $I^+$\emph{-anti-democratic}) iff there exists $b \in I^+$, $b$ is $F$-irrelevant. 

We say that $F$ is \emph{anti-democratic w.r.t. $I^*$} (or simply $I^*$\emph{-anti-democratic}) iff there exists $b \in I^*$, $b$ is $F$-irrelevant. 
\end{definition}
Definition \ref{democratic} essentially asserts that a collective choice $F$ is anti-democratic if it ignores a coalition of voters which is considered not small.

\begin{example}
A social choice function $F$ satisfying the Silver property is $\Finite^+$-anti-democratic, where $\Finite := \{ b \subseteq \N: b \text{ is finite} \}$.
\end{example}

\begin{example}
Let $I:=\{ \{ n \}: n \in \N   \}$, which means $I$ is the family consisting only of all singletons, and let $F$ be a social choice function. If $F$ satisfies dictatorship (in the Arrowian sense), then $F$ is $I^*$-anti-democratic. Hence dictatorship in the Arrowian sense can be viewed as a particular case of anti-democratic behaviour as meant by Definition \ref{democratic}.
\end{example}

\subsection{A brief digression regarding the Baire property.}
We now make a short digression by considering also another case, even if we will not focus on this line of research on this paper, but we consider it worth mentioning. 
Let $F \subseteq 2^{\N}$ be a set satisfying the Baire property, and as above interpret $F$ as a social choice function.
Then there exists $t \in 2^{<\N}$ such that $N_t \cap F$ is comeager or $N_t \cap F^c$ is comeager.
So in other words, we can find $t \in 2^{<\N}$ such that:
\begin{itemize}
\item for \emph{almost all} $x \in \cantor$ extending $t$, $F(x)=1$, or 
\item for \emph{almost all} $x \in \cantor$ extending $t$, $F(x)=0$.
\end{itemize}
The term ``almost all" above means for ``comeager many", as the filter of comeager sets is here interpret as a notion of largeness, in a similar fashion as events with probability measure 1 are considered to happen ``almost certainly" (in that case by using the filter of measure 1 sets instead of the comeager filter).  
In line with this example, one can generalize the above definition as follows.

\begin{definition} \label{irrelevant2}
Let $\filter{H} \subseteq \Power(2^{\N})$ be a filter.  Given a social choice function $F: 2^{\N} \rightarrow \{  0,1 \}$, $b \subseteq \N$ and $f : \N \rightarrow \{ 0,1 \}$ partial function with $\dom(f)=b^c$, we say that $b \subseteq \N$ is $\filter{H}$\emph{-almost} $(F,f)$-\emph{irrelevant} iff there exists $B \in \filter{H}$ for every $y,z \in N_f \cap B$, one has $F(z)=F(y)$.

We say that $b \subseteq \N$ is \emph{$\filter{H}$-almost $F$-irrelevant} iff there exists a partial function $f: \N \rightarrow \{0,1\}$ with $\dom(f)=b^c$ such that $b$ is $\filter{H}$-almost $(F,f)$-irrelevant.
\end{definition}

\begin{definition} \label{democratic2}
Let $\filter{H} \subseteq \Power(2^{\N})$ be a filter, and $F$ a social choice function.

We say that $F$ is  $(\filter{H},I^+)$\emph{-anti-democratic} iff there exists $b \in I^+$, $b$ is $\filter{H}$-almost $F$-irrelevant.

We say that $F$ is  $(\filter{H},I^*)$\emph{-anti-democratic} iff there exists $b \in I^*$, $b$ is $\filter{H}$-almost $F$-irrelevant.
\end{definition}

From the above observation, a social choice function $F$ satisfying the Baire property is $(\filter{C},I^*)$-anti-democratic, where $I:= \Finite := \{ b \subseteq \N: b \text{ is finite} \}$ and $\filter{C}$ denotes the comeager filter.

Note that one of the main limitations of our example above about the Baire property concerns the notion of ``almost all" associated with the comeager filter. In fact, as we mentioned above, the notion of almost-all depends on a given filter and in some cases one might also get a rather unnatural interpretation. For instance, we can consider the well-known example of a comeager set which has measure zero. In this case it looks not so reasonable to consider such a set necessarily a large set.

\section{Dense Silver sets} \label{s4}

In this section we would like to avoid the objection from remark \ref{remark2} and consider a more appropriate ideal $I$ (as $\Finite$ exhibits the limitations suggested in Remark \ref{remark2}).
We show that the notion of \emph{density} associated with a subset of natural numbers can well serve for our purpose.  
\begin{definition} \label{density}
Given $a \subseteq \N$, let $\alpha_n := \frac{|a \cap [0,n]|}{n}$ and define the following:
\begin{itemize}
\item the \emph{density of $a$} as $d(a):= \lim_{n < \N} \alpha_n$
\item the \emph{upper density of $a$} as $\overline{d}(a):= \limsup_{n < \N} 	\alpha_n$
\item the \emph{lower density of $a$} as $\underline{d}(a):= \liminf_{n < \N} \alpha_n$.
\end{itemize}
\end{definition}
It is easy to check that $d$ is not defined for every set, whereas upper and lower density are always defined. Moreover $d$ is defined exactly when $d(a)=\overline{d}(a)=\underline{d}(a)$. 
Throughout this section we consider the family
\[
D_\delta:= \{ b \subseteq \N: \overline{d}(b)\leq \delta \}.
\]
We use the notation $D:=D_0$ and note that $D$ is an ideal. 

We also introduce the notion of a tree and the standard notation associated, as it is useful to develop the arguments throughout this section.
\vspace{1mm}

\textbf{Trees.} Given $s,t \in 2^{<\omega}$ the notation $s \subseteq t$ means $s$ is an initial segment of $t$, $s \subset t$ means $s$ is a proper initial segment, $s \perp t$ are incompatible (i.e. neither $s \subseteq t$ not $t \subseteq s$), and $|t|:=|\dom(t)|$ denote the length of $t$. We say that $T \subseteq 2^{<\N}$ is a tree iff $T$ is closed under initial segments, i.e. $\forall t \in T \forall s \subseteq t (s \in T)$.
We call any $t \in T$ a \emph{node of $T$}. We also use the following standard notation:
\begin{itemize}
\item the set of \emph{terminal nodes of $T$} $\term(T) := \{ t \in T: \neg \exists s \in T (t \subset s)\}$;
\item the \emph{height of $T$} as $\height(T):= \sup \{ |t|: t \in  T \}$;
\item the \emph{body of $T$} as $[T]:=  \{ x \in 2^{\N}: \forall n \in \N (x \restric n \in T) \}$;
\item the set of \emph{splitting nodes of $T$} $\splitting(T):= \{ t \in T: t^\conc 0 \in T \land t^\conc 1 \in T   \}$;
\item the set of \emph{splitting levels of $T$} $\lev(T):= \{ |t|+1: t \in \splitting(T)  \}$.
\item for $t \in\splitting(T)$, $j \in \{ 0,1 \}$, put $\text{SplSucc}(t^\conc \langle j \rangle)$ be the shortest splitting node extending $t^\conc \langle j \rangle$.

\end{itemize}
Note that for any $N_f \in \silver$ there exists an infinite tree $T$ such that $[T]=N_f$. These are called Silver trees (see \cite{BLH2005}). Note that $\dom(f)^c=\lev(T)$.
\begin{definition}
Let $\delta \in [0,1]$. We say that $N_f$ is a \emph{$\delta$-dense Silver condition} (and write $N_f \in \uppersilver$) iff $N_f$ is a Silver condition and $\overline{d}(\dom(f)^c) \geq \delta$.

We say that $N_f$ is an \emph{upper dense Silver condition} (and write $N_f \in \silver^*$) iff $N_f$ is a Silver condition and $\overline{d}(\dom(f)^c)>0$.
\end{definition} 
Note that $\silver_0 = \silver$. Like for $\silver$, for every $N_f \in \uppersilver$ there exists a tree $T$ such that $[T]=N_f$.
As above we can analogously define the notion of a $\uppersilver$-set. 
A standard argument shows that AC implies the existence of non-$\uppersilver$-sets and non-$\silver^*$-sets. We now want to show the non-constructive nature of a non-$\uppersilver$-set.
Use the following notation
\begin{itemize}
\item[] $\nb$ := \text{ there is a non-Baire set}
\item[] $\nv$ := \text{ there is a non-$\silver$ set}
\item[] $\nv_{\delta}$ := \text{ there is a non-$\uppersilver$ set}
\item[] $\nuv$ := \text{ there is a non-$\silver^*$ set}

\end{itemize}

Note that if a social choice function $F$ is a non-$\silver^*$ set, then $F$ is not $D^+$-anti-democratic. In particular the example given by the sparse sets as considered in Remark \ref{remark2} does not occur any longer. In fact, given any social choice function $F$ which is a $\silver^*$-set, and $b \subseteq \omega$, if $b$ is $F$-irrelevant, then $b \in D^+$ per definition. Note that if we consider the case of a $\silver_1$ social choice function, then the irrelevant set $b$ has full upper density, i.e. $\overline{d}(b)=1$.

\begin{lemma} \label{silver-cohen-2}
Fix $\delta \in [0,1]$. For every comeager set $C$ there is $N_f \in \uv$ such that $N_f \subseteq C$.
\end{lemma}
\begin{proof}
Fix $\delta \in [0,1]$. Let $D_n: n \in \N$ be a $\subseteq$-decreasing sequence of open dense sets such that $\bigcap_{n \in \N} D_n \subseteq C$. 
Recall that if $D$ is open dense, then $\forall s \in 2^{<\N}$ there exists $s^{\prime} \supseteq s$ such that $N_{s^{\prime}} \subseteq D$.
We build $T \in \silver_\delta$ by recursively constructing its nodes as follows.
\begin{itemize} 
\item First pick $t_{\emptyset} \in 2^{<\N}$ such that $N_{t_\emptyset} \subseteq D_0$, and then let 
\begin{align*}
F_0:=& \bigcup \{{t_\emptyset} ^\conc s: s \in 2^{|t_\emptyset|} \}\\
T_0:=& \{ t' \in 2^{<\N}: \exists t \in F_0 (t' \subseteq t) \}.
\end{align*}
\item Assume $F_n$ and $T_n$ already defined. 
Let $\{t_j: j \leq J \}$ enumerate all nodes in $F_n$, which are all terminal nodes in $T_n$. We proceed inductively as follows: pick $r_0 \in 2^{<\N}$ such that $N_{t_0^\conc r_0} \subseteq D_{n+1}$; then pick $r_1 \supseteq r_0$ such that $N_{t_1^\conc r_1} \subseteq D_{n+1}$; proceed inductively in this way for every $j \leq J$, so $r_j \supseteq r_{j-1}$ such that $N_{t_j^\conc r_j}\subseteq D_{n+1}$. 
Finally put $h=n \cdot  |t^\conc r_J|$ for $t \in \term(T_n)$.
Then define 
\begin{align*}
F_{n+1}:=& \bigcup \{ t^\conc {r_J}^\conc s: t \in \term(T_n), s \in 2^{h} \} \\
T_{n+1}:=& \{ t' \in 2^{<\N}: \exists t \in F_{n+1} (t' \subseteq t)\}.
\end{align*}
Note that by construction for all $t \in F_{n+1}$ we have $N_t \subseteq D_{n+1}$. Moreover by the choice of $h$ it follows that 
\begin{equation} \label{eq1}
\frac{|\text{Lev}(T_{n+1})|}{|\text{ht}(T_{n+1})|} \geq \delta  ( 1 - \frac{1}{n} ).
\end{equation}
\end{itemize} 
Finally put $T := \bigcup_{n \in \N} T_n$. By construction $[T] \in \silver_\delta$ as \ref{eq1} implies that $\overline{d}(\lev(T))=\delta$. 

It is left to check $[T] \subseteq \bigcap_{n \in \N} \in D_n$. To show that, fix arbitrarily $x \in [T]$ and $n \in \N$. By construction there is $t \in F_n$ such that $t \subset x$ and since $N_t \subseteq D_n$ we then get $x \in N_t \subseteq D_n$. 
\end{proof}

\begin{corollary}
$\nv \Rightarrow \nuv \Rightarrow \nv_{\delta} \Rightarrow \nb$
\end{corollary}

\begin{proof}
The first two implications are immediate. For the third one $\nv_{\delta} \Rightarrow \nb$ just recall that if $F$ satisfies the Baire property, then $F$ is meager or there exists $t \in 2^{<\omega}$ such that $N_t \cap F$ is comeager. In the former case, Lemma \ref{silver-cohen-2} gives $T  \in \silver_\delta$ such that $[T] \cap F = \emptyset$, while in the latter case the Lemma gives $T \in \silver_\delta$ such that $[T]\subseteq F$ (with $\stem(T) \supseteq t$). 
\end{proof}

So in Shelah's model where all sets have the Baire property, we also get $\neg \nv_{\delta}$. Moreover the following two lemmata show that the existence of non-$\uppersilver$ sets requires a strictly larger fragment of AC than the existence of non-Baire sets, and indeed we can build a model where all sets satisfy the $\uppersilver$-property, but there is a $\SSigma^1_2$ non-Baire set. 
\begin{lemma} \label{silver-cohen}
Let $\delta \in [0,1]$ and $T \in \silver_{\delta}$. Let $\bar \varphi: \splitting(T) \rightarrow 2^{<\omega}$ such that $\bar \varphi(\stem(T)):=\langle \rangle$ and for every $t \in \splitting(T)$ and $j \in \{  0,1 \}$, $$\bar \varphi(\text{SplSucc}(t^\conc \langle j \rangle)):= \bar \varphi(t)^\conc \langle j \rangle.$$ Put
$\varphi: [T] \rightarrow 2^\omega$ be the expansion of $\bar \varphi$, i.e. for every $x \in [T]$, $\varphi(x):= \bigcup_{n \in \omega} \bar \varphi(t_n)$, where $t_n:n \in \omega$ is a $\subseteq $-increasing sequence of splitting nodes in $T$ such that $x=\bigcup_{n \in \omega} t_n$. 

If $c$ is Cohen generic over $V$, then
\[
V[c] \models \exists T' \in \silver_{\delta} \land T' \subseteq T \land \forall x \in [T'] (\varphi(x) \text{ is Cohen over $V$}),
\]
\end{lemma}
\begin{proof}
Consider the following forcing $\poset{P}$ consisting of finite trees $p \subseteq 2^{<\N}$ such that
$\forall s,t \in \term(p)$, $|s|=|t|$
ordered by end-extension: $p' \leq p$ iff $p' \supseteq p$ and $\forall t \in p'\setminus p \exists s\in \term(p) (s \subseteq t)$. 

Note $\poset{P}$ is countable and non-trivial, thus it is equivalent to Cohen forcing $\cohen$.
Let $T_G:= \bigcup G$, where $G$ is $\poset{P}$-generic over $V$. We claim that $T':= \bar \varphi^{-1}"T_G$ satisfies the required properties. It is left to show that:
\begin{enumerate}
\item for every $x \in [T']$ one has $\varphi(x)$ is Cohen, i.e., every $y \in [T_G]$ is Cohen;
\item $T' \in \silver_\delta$.
\end{enumerate}

For proving (1), let $D$ be an open dense subset of $\cohen$ and $p \in \poset{P}$. It is enough to find $p' \leq p$ such that every $t \in \term(p')$ is a member of $D$.

The idea is the same as in the proof of Lemma \ref{silver-cohen-2}.
Let $\{ t_j: j < N \}$ enumerate all terminal nodes in $p$ and pick $r_J$, so that for every $j < N$, $N_{t_j^\conc r_J} \subseteq D$. Then put $p':= \{t \in 2^{<\omega}: \exists t_j  \in \term(p) (t \subseteq t_j^\conc r_J \}$. Hence $p' \leq p$ and $p' \force \forall y \in [T_G] \exists t \in p' \cap D (t \subset y)$. Hence we have proven that 
\[
\force_\cohen \forall y \in [T_G] \exists t \in  D (t \subset y),
\]
which means every $y \in [T_G]$ is Cohen over $V$.

For proving (2), it is enough to note that, given $k \in \N$ and $p \in \poset{P}$ arbitrarily, one can find $q \leq p$ such that
 
\begin{equation} \label{eq2}
\frac{|\text{Lev}(\bar \varphi^{-1}"q)|}{\text{ht}(\bar \varphi^{-1}"q)} \geq \delta (1 - \frac{1}{2^{k}}).
\end{equation}
That means for every $k \in \omega$ the set $E_k$ consisting of $q$'s as in (\ref{eq2}) is dense, and so $\Lev(T')$ has upper density $\geq \delta$.
\end{proof}

\begin{proposition} \label{cohen-uppersilver}
Let $G$ be $\cohen_{\omega_{1}}$-generic over $V$. There is an inner model $M$ of $V[G]$ such that
\[
M \models \forall F \subseteq 2^{\N} \forall N_f \in \silver_\delta \exists N_g \leq N_f (N_g \subseteq F \vee N_g \cap F=\emptyset).
\]
(In particular, in $M$ every $F \subseteq 2^\omega$ satisfies the $\silver_\delta$-property.)
\end{proposition}
\begin{proof}
The argument is standard and is essentially the same as in the proof of \cite[Proposition 3.7]{BLH2005}. For the reader convenience and for the sake of completeness, we give some details.
Fix $\delta \in [0,1]$. 
Let $G$ be $\bC_{\omega_1}$-generic over $V$. 

Let $F$ be an $\text{On}^\omega$-definable set of reals, i.e. $F:= \{x \in 2^\omega: \Phi(x,v)   \}$ for a formula $\Phi$ with parameter $v \in \text{On}^\omega$, and let $N_f \in \silver_\delta$.  We aim to find $N_g \leq N_f$ such that $N_g \subseteq F$ or $N_g \cap F=\emptyset$. 

First note that we can absorb $v$ and $N_f$ in the ground model, i.e, we can find $\alpha<\omega_1$ such that $v, N_f \in V[G \restric \alpha]$. Let $\varphi: N_f \rightarrow 2^\omega$ be as in Lemma \ref{silver-cohen}.
Let $c= G(\alpha)$ be the next Cohen real and write $\bC$ for
the $\alpha$-component of $\bC_{\omega_1}$.

There are $b_0 =  \big \llbracket \llbracket (\Phi(\varphi^{-1}(c),v)) \rrbracket_{\cohen_{\alpha}} = \mathbf{0}  \big \rrbracket_\cohen$ and
    $b_1 =  \big \llbracket \llbracket
    \Phi(\varphi^{-1}(c),v))\rrbracket_{\cohen_{\alpha}}= \mathbf {1}  \big \rrbracket_\cohen$, and by $\cohen$-homogeneity $b_0 \land b_1 = \0$ and $b_0 \vee b_1=\1$.
Hence, by applying Lemma \ref{silver-cohen}, one can then find $N_g \leq N_f$ such that $N_g \subseteq b_0$ or $N_g \subseteq b_1$ and for every $x \in N_g$, $\varphi(x)$ is Cohen over $V[G \restric \alpha]$. We claim that $N_g$ satisfies the required property.
\begin{itemize}
\item Case $N_g \subseteq b_1$: note for every $x \in N_g$, $\varphi(x)$ is Cohen over $V[G \restric \alpha]$, and so $V[G\restric \alpha][\varphi(x)] \models \llbracket \Phi(\varphi^{-1}(\varphi(x)),v)=\mathbf{1} \rrbracket_{\cohen_{\omega_1}}$. Hence $V[G] \models \forall x \in N_g(\Phi(x,v))$, which means $V[G] \models N_g \subseteq F$.  
\item Case $N_f \subseteq b_0$: we argue analogously and get $V[G] \models \forall x \in N_g(\neg \Phi(x,v))$, which means $V[G] \models N_g \cap F = \emptyset$.
\end{itemize}
\end{proof}

The limitations showed in this section actually provides also rather negative insights about the definability of non-democratic collective functions. 
As we know that in ZFC one can prove that every $\SSigma^1_1$ set satisfies the Baire property, and due to $\nuv \Rightarrow \nb$ we then obtain the following.
\begin{corollary}
Every $\SSigma^1_1$ social choice function is $D^+$-anti-democratic. 
\end{corollary}
Even if in a different setting, this expectation is in line with Mihara's \cite{Mihara1} and \cite{Mihara2}, where it is proven that social choice function \emph{\'a la Fishburg} cannot have a recursive nature. Hence, the corollary above should be understood as another signal that non-democratic social choice functions seem to be hardly definable. 

Another consequence we get is the following. We go back to analyse the role of AD in this context in the concluding remarks.

\begin{corollary} \label{corollary1}
Assume AD. Then every social choice function is $D^+$-anti-democratic.
\end{corollary}

\section{More than two alternatives} \label{s5}

We now analyse the case when each individual can choose among a set with more than two alternatives. In particular we distinguish the case when there are finitely many alternatives, i.e. $A:= \{ a_i: i< K\}$, with $K>2$ finite, and $\N$ alternatives, i.e. $A:=\{ a_i: i \in \N \}$.

\subsection{Finitely many alternatives}
We start with the former case. So w.l.o.g  let $A=K$ (for some finite $K>2$) be the set of alternatives that each individual $i \in \N$ can choose, i.e. $x(i) \in K$. We can develop the notion of Silver-like conditions in a similar fashion as in the previous sections; given $f: \N \rightarrow K$ partial function, let $N_f:=\{ x \in K^{\N}: \forall n \in \N (x(n)=f(n)) \}$. Moreover we can naturally generalizes the notion of Silver and dense Silver forcing in this context, and we refer to them as $\silver_K$ and $\uppersilver_K$.

In this framework, we think of a social choice function as an $F: K^{\N} \rightarrow K$, i.e. $F$ selects the collectively preferred candidate in $K$ for every given sequence $x \in K^{\N}$ of individual choices. In this framework with more than two alternatives, the suitable version of Silver property generalises as follows. 

\begin{definition} \label{gen-silver-property}
Let $F: K^{\N} \rightarrow K$ be a social choice function, and $H:=\{ h_i: i<k \}$, with $k \leq K$, be a family of pairwise disjoint subsets of $K$.

We say that $F$ satisfies the \emph{$(\uppersilver_K,H)$-property} iff for every $N_f \in \uppersilver_K$ there exists $N_g \leq N_f$ and $h \in H$ such that $\forall x \in N_g$, $F(x) \in h$.

We say that $F$ satisfies the \emph{$\uppersilver_K$-property} iff for every $N_f \in \uppersilver_K$ there exists $N_g \leq N_f$ and $k \in K$ such that $\forall x \in N_g$, $F(x) = k$.
\end{definition}

We can easily check that Lemma \ref{cohen-uppersilver} also holds for the $\uppersilver_K$-property. 
Fix arbitrarily $K > 2$. The idea simply relies on the following recursive argument. Let $F: K^\omega \rightarrow K$ be a social choice function and $N_f \in  \uppersilver_K$. Start with $H_0:=\{ h_0, h'_0  \}$ partition of $K$; we can use the same argument as in Lemma \ref{cohen-uppersilver} in order to obtain $N_{f_0} \leq N_{f}$ such that 
\[
\forall x \in N_{f_0} (F(x)\in h_0 ) \vee \forall x \in N_{f_0} (F(x) \in h'_0).
\]
W.l.o.g. assume the former case occurs and pick $h_0$ (otherwise we use the same argument by picking $h'_0$). Then pick $H_1:=\{ h_1,h'_1 \}$ partition of $h$ and use again the same argument as in the proof of Lemma \ref{cohen-uppersilver} in order to find $N_{f_1} \leq N_{f_0}$ such  that  
\[
\forall x \in N_{f_1} (F(x)\in h_1 ) \vee \forall x \in N_{f_1} (F(x) \in h'_1).
\]
It is clear that one can recursively proceed similarly as soon as we find $N_{f_m} \leq N_{f}$ and $h_m$ consisting of a single element $k \in K$, and so $\forall x \in N_{f_m}$, one has $F(x)=k$, as desired.

Also, an analog of Lemma \ref{silver-cohen-2} holds in this case with $K$ alternatives.

\subsection{Infinitely many alternatives} 

Let $A=\N$ and define $N_f:=\{ x \in \N^{\N}: \forall n \in \N (x(n)=f(n)) \}$, for a given partial function $f:\N \rightarrow \N$, and consider the analog Silver-like forcing in $\N^{\N}$ denoted by $\silver_\infinite$. Definition \ref{gen-silver-property} naturally generalises in the infinite case as well, simply by replacing the natural number $K$ with $\N$. 
However the situation with infinitely many alternatives gives rise to very different results, as we show in this section.

Trees associated with conditions in $\silver_\infty$ are a particular kind of the following forcing notion, which has been studied in \cite{KL2017}: we say that a tree $T \subseteq \omega^{<\omega}$ is a \emph{full-Miller tree} (in symbol $T \in \poset{FM}$) iff for every $s \in T$ there exists $t \supseteq s$ such that for all $n \in \omega$, $t^\conc n \in T$. 

Note that $\silver_\infty \subseteq \poset{FM}$. The forcing $\poset{FM}$ adds Cohen reals (see \cite{KL2017}) and the same argument shows $\silver_\infty$ adds Cohen reals. 
As a consequence, in line with the proof of \cite[Proposition 4.4]{LSR2020}, one can prove $\nb \Rightarrow \nv_\infty$, where $\nv_\infty$ is the statement asserting that there exists a non-$\silver_\infty$-set, in line with the notation introduced in Section \ref{s4}.

Moreover Brendle showed that Lemma \ref{silver-cohen} does not hold in such a case, which means that Cohen forcing $\cohen$ cannot add a tree $T \in \poset{FM}$ (and \emph{a fortiori} cannot add $T \in \silver_\infty$). 
\begin{proposition}[Brendle, \cite{Brendle2}] \label{cohen-silver-inf} 
$\cohen$ does not add a condition $N_f \in \silver_\infinite$ of Cohen reals.
\end{proposition}

(Brendle actually proves this result  for the Miller forcing and so in particular holds for $\silver_\infty$.)

But $\poset{FM}$ and $\silver_\infty$ are different under another point of view, which is particularly significant in our framework of social choice theory. In fact, while for $\poset{FM}$ an analog of Lemma \ref{silver-cohen-2} holds, for $\silver_\infty$ it fails.

\begin{lemma} \label{comeager-silver-inf}
There exists a comeager set $C \subseteq \omega^\omega$ such that for every $N_f \in \silver_\infinite$, $N_f \nsubseteq C$.
\end{lemma}

\begin{proof}
Let $e_n$ denote the constantly equal 0 sequence of length $n$, i.e. $|e_n|=n$ and $\forall i < n (e_n(i)=0)$. We define a $\subseteq$-decreasing sequence $\{ C_n: n \in \N \}$ of open dense sets as follows:
\begin{itemize}
\item let $h_0:\N^{<\N} \rightarrow \N^{<\N}$ be such that, for every $s \in \N^{<\N}$, $j \in \N$, let $h_0(s^\conc \langle j \rangle):= {s}^\conc {e_{s(0)}}^\conc \dots^\conc {e_{s(|s|-1)}}^\conc  {e_j}$ (note that for $s = \langle \rangle$ we have $h_0(\langle j \rangle)=e_j$). Then put 
\[
C_0:= \bigcup_{s \in \N^{<\N}} N_{h_0(s)}.
\]
\item let $h_{n+1}:\N^{<\N} \rightarrow \N^{<\N}$ be such that, for every $s \in \N^{<\N}$, $j \in \N$, let $h_{n+1}(s^\conc \langle j \rangle):={s}^\conc {e_{s(0)}}^\conc \dots^\conc {e_{s(|s|-1)}}^\conc {e_j}^\conc {e_{n+1}}$, and then put 
\[
C_{n+1}:= \bigcup_{s \in \N^{<\N}} N_{h_{n+1}(s)}.
\]
\end{itemize} 
Note that for every $s \in \N^{<\N}$, $h_{n+1}(s) \supseteq h_{n}(s)$, and so $C_{n+1} \subseteq C_n$.

By construction, each $C_n$ is open dense. Put $C := \bigcap_{n \in \N} C_n$. We aim to show:
\begin{equation} \label{eq4}
\forall N_f \in \silver_\infinite \exists x \in N_f \setminus C,
\end{equation}
and so we have to show that for every $N_f \in \silver_\infinite$ there exist $x \in \omega^\omega$ and $n \in \N$ such that $x \in N_f$ and $x \notin C_n$. So fix $N_f \in \silver_\infinite$ arbitrarily and let $T_f$ be the corresponding tree such that $[T_f]=N_f$, and let $\{ a_j:j \in \omega \}$ enumerate all elements in $\dom(f)^c$. 
Let $t_0:= f \restric [0,a_0) = \stem(T_f)$ and pick $n > a_0+1$; fix the notation $E_m:=\bigcup_{s \in \N^{< m}} h_n(s)$. For every $s \in \omega^{<a_0}$ and $j \in \omega$ we get 
\[
h_n(s^\conc \langle j \rangle) \supseteq s^\conc e_n,
\]
and since, by the choice of $n$, for every $s \in \omega^{<a_0}$ and $k >0$ one has $s^\conc e_n \perp {t_0}^\conc \langle k \rangle$ we therefore get $E_{a_0} \cap N_{t_0^\conc \langle k \rangle}=\emptyset$.

Then take $j_1 > a_1 +1$ and put $t_1:= {t_0}^\conc \langle j_1 \rangle^\conc f \restric (a_0,a_1)$; in particular, note that $E_{k_0} \cap N_{t_1}=\emptyset$. By definition $h_n({t_0}^\conc \langle j_1 \rangle^\conc {\langle f(a_0+1) \rangle}) \supseteq {t_0} ^\conc {\langle j_1 \rangle}^\conc {e_{j_1}}^\conc {e_n}$, and similarly for every $ s \subseteq  f \restric (a_0,a_1)$, we have
\[
h_n({t_0}^\conc \langle j_1 \rangle^\conc {s}) \supseteq {t_0} ^\conc {\langle j_1 \rangle}^\conc {s \restric_{|s|-1}}^\conc {e_{s(|s|-1)}}^\conc {e_{j_1}}^\conc {e_n}
\]

Hence, since $|e_{j_1}|=j_1 > a_1 +1> |f\restric (a_0,a_1)|$, it follows, by the same argument as above, for every $k >0$, $E_{a_1} \cap N_{t_1^\conc k}=\emptyset$.

Then  we can proceed following the same procedure and recursively build the sequences $\{t_m: m \in \N  \}$ such that $t_m \subseteq t_{m+1}$ and for every $m \in \omega$, $N_{t_{m+1}} \cap E_{a_m}= \emptyset$. Finally put $x := \bigcup_{m \in \N} t_m$. By construction each $t_m \in T_f$ and so $x \in [T_f]=N_f$. We claim 
$$x \notin \bigcup_{m \in \N} E_{a_m}.$$
Indeed, assume there is $l \in \omega$ such that $x \in E_{a_l}$; then we would have $\bigcap_{m \in \omega} N_{t_m}=x \in E_{a_l}$, contradicting the fact that $N_{t_{l+1}} \cap E_{a_l} = \emptyset$.

The proof is then complete, as $ \bigcup_{m \in \N} E_{a_m}=C_n$. 
\end{proof}

\begin{proposition} \label{non-silver-inf}
There exists $F \subseteq \N^{\N}$ such that for every $N_f \in \silver_\infinite$ one has $N_f \cap F \neq \emptyset$  and $N_f \nsubseteq F$.  
\end{proposition}

If we identify $F$ with its characteristic function, Proposition \ref{non-silver-inf} asserts that:
\begin{equation} \label{eq5}
\forall N_f \in \silver_\infinite (\exists x_0 (F(x_0)=0) \land \exists x_1 (F(x_1)=1). 
\end{equation}

Note that this is not enough to get a social choice function for infinite alternatives which is not $\Finite^+$-anti-democratic, as the $F$ is \ref{eq5} only range into $\{  0,1\}$ and not $\N$. However, we conjecture that Proposition \ref{non-silver-inf} might be improved for $F: \N^{\N} \rightarrow \N$. 

\begin{proof}[Proof of Proposition \ref{non-silver-inf}]
The construction is very similar to that one in the proof of Lemma \ref{comeager-silver-inf}. The difference here is that we build the $h_n$'s in order to work as maps from $\silver_\infinite \rightarrow \silver_\infinite$ instead of maps from $\N^{<\N}$ to $\N^{<\N}$. For ease of notation we write $f \in \silver_\infty$ since the condition $N_f$ is completely determined by its associated partial function $f$. We also use the following notation: given a partial function $f: \N \rightarrow \N$ with co-infinite domain, let $\{ a_j:j \in \omega \}$ enumerate all elements in $\dom(f)^c$, and $t \in \N^{<\N}$ we define the partial function $f \oplus t \in \silver_\infinite$ as follows:
\begin{equation}
(f \oplus t )(n)=
\begin{cases}
f(n) \quad &\text{if}\quad n \in \dom(f) \\ 
t(j) \quad &\text{if}\quad  n=a_j \land j < |t|.
\end{cases}
\end{equation}
Clearly $N_{f \oplus t} \leq N_f$. 
Moreover, given $\{ t_i: i \leq N \}$, with natural number $N>0$, we also recursively define $f \bigoplus_{i \leq N} t_j:= (f \bigoplus_{i<N} t_i) \oplus t_N$.
We then define $G_n$'s recursively as follows:  
\begin{itemize}
\item let $G_0: \silver_\infinite  \rightarrow \silver_\infinite$ be such that, for every $f \in \silver_\infinite$, with $\dom(f)^c:=\{ a^f_j:j \in \N \}$, let 
\[
G_0(f \cup \{ a^f_0,j \}):=  (f \bigoplus_{i < a^f_0} e_{f(i)}) \oplus e_j. 
\] 
(Note that for if $a^f_0=0$ we get $G_0(f \cup \{ a^f_0,j \}):=f \oplus e_j$; and if $f = \langle \rangle$ we get $G_0(f \cup \{ a^f_0,j \}):=e_j$). Then put 
\[
F_0:= \bigcup \{ [G_0(f \cup \{ a^f_0,j \})]: f \in \silver_\infinite \land j \in \N  \}.
\]
Note that $F_0$ is defined in such a way that it is $\silver_\infinite$-open dense, which means
\[
\forall N_f \in \silver_\infinite \exists N_g \leq N_f (N_g \subseteq F_0).
\]
\item for $n >0$, let $G_n: \silver_\infinite  \rightarrow \silver_\infinite$ be such that, for every $f \in \silver_\infinite$, with $\dom(f)^c:=\{ a^f_j:j \in \N \}$, let 
\[
G_n(f \cup \{ a^f_0,j \}):=  (f \bigoplus_{i < a^f_0} e_{f(i)}) \oplus e_j \oplus e_n. 
\] 
Then put 
\[
F_n:= \bigcup \{ [G_n(f \cup \{ a^f_0,j \})]: f \in \silver_\infinite \land j \in \N  \}.
\]
Note that $F_n$ is $\silver_\infinite$-open dense, and $F_{n+1} \subseteq F_n$.
\end{itemize}
Finally put $F:= \bigcap_{n \in \N} F_n$. We claim that $F$ has the desired property. 

Fix arbitrarily $N_f \in \silver_\infinite$. 
We can build a sequence $\{ t_m: m \in \N \}$ of nodes in $N_f$ following the same procedure as in Lemma \ref{comeager-silver-inf}, pick $n > a^f_0 + 1$ and then define $y:= \bigcup_{m \in \N} t_m$ such that $y \in N_f$. In this case, by construction we have $t_{m+1} \notin A_m$, where $A_m$ is defined as
\[
A_m:= \bigcup \{ [G_n(N_g)]: N_g \in \silver_\infinite \land a^g_0 < a^f_m \}.
\]
Since $F_n = \bigcup_{m \in \N} A_m$ we the get $y \notin F_n$. This complete the proof that $N_f \not \subseteq F$. 

Showing that $N_f \cap F \neq \emptyset$ is even simpler, as it is sufficient to define the following:
\begin{equation}
x(n):=
\begin{cases}
f(n) \quad &\text{if}\quad n \in \dom(f) \\ 
0 \quad &\text{if}\quad  \exists j \in \omega (n=a^f_j).
\end{cases}
\end{equation}
It is clear that $x \in N_f$. To check that $x \in F$, it is enough to note that 
\[
x \in \bigcap_{n \in \omega} [G_n (f \cup \{ a^f_0,0 \})],
\]  
and since $ [G_n (f \cup \{ a^f_0,0 \})] \subseteq F_n$, we obtain $x \in F$.

\end{proof}

\section{Equity and Pareto principles vs anti-democratic social choice functions} \label{s6}

A well-studied field in economic theory consists of social welfare relations on infinite utility streams. 

We consider a \emph{set of utility levels} $Y$ (or \emph{utility domain}) totally ordered, and we call $X:= Y^\omega$ the corresponding \emph{space of infinite utility streams}. In some setting $Y$ is usually also endowed with some topology and $X$ with the induced product topology, but in our context we do not need any topological setting.
Given $x,y \in X$ we write $x \leq y$ iff $\forall n \in \omega (x(n) \leq y(n))$, and $x < y$ iff $x \leq y \land \exists n \in \omega (x(n) < y(n))$. Furthermore we set $\ideal{F}:= \{ \pi: \omega \rightarrow \omega: \text{ finite permutation} \}$, and we define, for $x \in X$, $f_\pi(x):= \langle x({\pi(n)}): n \in \omega \rangle$.

We say that a subset $\preceq$ of $X \times X$ is a \emph{social welfare relation} (SWR) on $X$ iff $\preceq$ is reflexive and transitive, and let $\sim$ denote the equivalence relation $x \sim y \ifif x \preceq y \land  y \preceq x$. If $x \preceq y$ and $x \not \sim y$ we write $x \prec y$.

In economic theory some equity and efficiency principles are well-studied; we focus on three of the most popular ones.

\begin{definition} \label{def-FA}
A social welfare relation $\preceq$ is called \emph{finitely anonymous} (FA) (or satisfying finite anonymity)  iff for every $\pi \in \ideal{F}$ we have $f_\pi(x) \sim x$. 
\end{definition}

\begin{definition}
A social welfare relation satisfies \emph{strong equity} (SE) iff for every $x,y \in X$ such that there exist $i,j \in \omega$, $i \neq j$, such that $x(i) < y(i) < y(j) < x(j)$ and for all $k \neq i,j$, $x(k)=y(k)$, then $x \prec y$.
\end{definition}

\begin{definition} \label{def-pareto}
A social welfare relation $\preceq$ is called \emph{Paretian} (P) (or satisfying Pareto principle) iff 
$\forall x,y \in X (x < y \Rightarrow x \prec y)$.
\end{definition}

The first two are considered equity principles: FA essentially means that when we interchange the positions of two (or finitely many) individuals the SWR does not change the ranking; SE can be understood as a principle saying that if we consider an income distribution $x$ and we increase the income of a poor person $i$ and decrease the income of a rich person $j$ (and leave the income of the others as they are) then we get $y$ which is better-off than $x$; P is understood as an efficiency principle, in fact if a stream $y$ is everywhere larger than another stream $x$ and it is strictly larger for at least one individuals, than $y$ is strictly preferred than $x$.

In recent years total SWR satisfying combinations of these principles were proven to have a non-constructive nature (\cite{Zame}, \cite{Lauwers1}, \cite{Dubey}, \cite{Laguzzi}, \cite{DLR}). We want to show that some combinations of these principles are in a sense able to provide social welfare functions which are not $D_{\delta}^+$-anti-democratic, for $\delta \geq \frac{2}{3}$ where we recall
\[
D_{\delta}:= \{ x \in 2^{\omega}: \overline{d}(x) \leq \delta  \}.
\]

In the following proposition, keeping in mind the interpretation of SE given above, we could think of the utility levels $a < b < c < d$ as: $a:=$ poor class, $b=$ lower-middle class, $c=$ upper-middle class, $d=$ rich class.
\begin{proposition}
Let $\preceq$ denote a SWR satisfying SE and FA on $X= Y^{\omega}$, where $Y:=\{ a,b,c,d  \}$ with $a < b < c <d$.
Then there exists a social choice function $F$ which is not $D_\delta^+$-anti-democratic, for any $\delta \in (\frac{2}{3},1]$.
\end{proposition}

\begin{proof}
We use the following notation. Given $f : \omega \rightarrow \{ 0,1 \}$ partial function, let $U(f):= \{ n \in \N: f(n)=1 \}$. Given $x \in 2^{\omega}$ let $\{ n_k: k \in \N \}$ enumerate $U(x)$,  $I_{k}:= [2 n_k, 2 n_{k+1})$ and
\[
O(x):= \bigcup_{k \in \odd} I_k  \quad \text{ and } \quad E(x):= \bigcup_{k \in \even} I_k.
\]
Then define $o(x), e(x) \in Y^\omega$ as follows
\begin{equation}\label{E1}
o(x)(n):=  
\begin{cases}
a		&\text{ if $n \in [0,n_0)\cap \even$}\\
d		&\text{ if $n \in [0,n_0)\cap \odd$}  \\
b		&\text{ if $n \in E(x)\cap \even$}\\
c		&\text{ if $n \in E(x)\cap \odd$}  \\
a		&\text{ if $n \in O(x)\cap \even$}  \\
d		&\text{ if $n \in O(x)\cap \odd$}\\
\end{cases}
\end{equation}
\begin{equation}\label{E2}
e(x)(n):=  
\begin{cases}
a		&\text{ if $n \in [0,n_0)\cap \even$}\\
d		&\text{ if $n \in [0,n_0)\cap \odd$}  \\
a		&\text{ if $n \in E(x)\cap \even$}\\
d		&\text{ if $n \in E(x)\cap \odd$}  \\
b		&\text{ if $n \in O(x)\cap \even$}  \\
c		&\text{ if $n \in O(x)\cap \odd$}  \\
\end{cases}
\end{equation}
Let $\preceq$ be a SWR satisfying SE and FA, and put 
\[
F:= \left\{ x \in 2^\omega: e(x) \prec o(x) \right\}.
\] 
We aim to show $F$ is a non-$\silver_{\delta}$-set. In particular, this means $F$ is the social choice function we were looking for, as a non-$\silver_\delta$-set is not $D_\delta^+$-anti-democratic, for any fixed $\delta \in (\frac{2}{3},1]$.

Given any $N_f \in \silver_\delta$,  let $\{n_k: k \in\N \}$ enumerate all natural numbers in $\dom(f)^c \cup U(f)$. 
Our goal is to find $x, z \in N_f$ such that $x \in F \ifif z \notin F$. 
First we observe an easy combinatorial structure of conditions in $\silver_{\delta}$.

\begin{claim}
Let $A \subseteq \omega$ be such that $\overline{d}(A)> \frac{2}{3}$. Then there are infinitely $w_j$ such that for every $j < \omega$, one has $w_j \subseteq A$, $|w_j|=3$ and there exist $h_j \in \omega$ so that $w_j=\{ h_j,h_j+1, h_j+2 \}$.  
\end{claim}

The proof of the claim is an easy counting-argument. Indeed, note that $\overline{d}(A) > \frac{2}{3}$ there are $\{ k_i: i < \omega \}$ such that for every $i < \omega$, $\frac{|A \cap {k_i}|}{{k_i}} > \frac{2}{3} + \varepsilon$, for some $\varepsilon> 0$. Then it is easy to see that any block $A \cap [{k_i},{k_{i+1}})$ should contain a triple as in the claim. 

\vspace{2mm}

We pick $x \in N_f$ such that for all $n_k \in \dom(f)^c \cup U(f)$, $x(n_k)=1$. 
Let $\left\{w_j: j \in \N \right\}$ list all triples as in the claim, applied with $A=\dom(f)^c$. 

We have to consider three cases.

\begin{itemize}
\item $e(x) \prec o(x)$: 
Note $x \in F$. 
Pick $n_{l} \in w_0$ such that $l \in \even$ and define $y \in 2^\omega$ as follows: 
\[
y(n) = \Big \{ 
\begin{array}{ll}
x(n) & \text{if $n \neq n_{l}$}\\
0 &  \text{if $n = n_l$}.
\end{array}
\]
Roughly speaking, $y$ is obtained by dropping $n_l$ from $x$.
Note we have
\begin{itemize}
\item[] $\forall n \geq n_l, o(y)(n)=e(x)(n)$
\item[] $\forall n \geq n_l, e(y)(n)=o(x)(n)$
\end{itemize}
Let
\[
O(l):= \bigcup_{k<l, k \in \odd} I_k \quad \text{ and } \quad E(l):= \bigcup_{k <l, k \in \even} I_k.
\]

\vspace{2mm}

Pick $J \in \omega$ large enough, $n_j,n_j+1 \in \dom(f)^c$ (for every $j \in J$) such that $j \in \odd$ and $2 \cdot J > |E(l)|+2$; note that $n_j+1 = n_{j+1}$ since we choose $n_j,n_{j}+1 \in w_{i_j}$, for some $i_j \in \omega$. Then define 
\[
z(n) = \Big \{ 
\begin{array}{ll}
y(n) & \text{if $n \notin \{ n_{j}, n_{j}+1: j<J \}$}\\
0 &  \text{otherwise}.
\end{array}
\] 
Roughly speaking, $z$ is obtained by dropping the $n_j$'s and $n_j+1$'s from $y$.

Note that by definition for every $j \in J$, $o(z)(2n_j)=a$ and $o(z)(2n_j+1)=d$.  Now let $\{ e_j: j < J' \}$ and $\{ o_j: j < J'\}$ be increasing enumerations, respectively, of all even numbers and odd numbers in $E(l)$; note that by choice of $J$ we have $J' < J$. Then we have the following.
\begin{itemize}
\item For all $j < J'$, one has $o(z)(e_j)=b=e(x)(2n_j)$ and $o(z)(o_j)=c=e(x)(2n_j+1)$, while  $e(x)(e_j)=a=o(z)(2n_j)$ and $e(x)(o_j)=d=o(z)(2n_j+1)$ . Now pick $\pi \in \mathcal{F}$ such that for every $j < J'$, $\pi(e_j)=2n_j$ and $\pi(o_j)=2n_j+1$; 
note that there exists $j \in J \setminus J'$ such that $2n_j \notin \dom(\pi)$ and $o(z)(2n_j)= a < e(x)(2n_j)=b < e(x)(2n_j+1)=c < o(z)(2n_j+1)=d$;
hence $o(z) \sim f_\pi(o(z)) \prec e(x)$.
\item Analogously, for all $j < J'$, one has $e(z)(e_j)=a=o(x)(2n_j)$ and $e(z)(o_j)=d=o(x)(2n_j+1)$, while $o(x)(e_j)=b=e(z)(2n_j)$ and $o(x)(o_j)=c=e(z)(2n_j+1)$. Pick the same $\pi \in \mathcal{F}$ as above, i.e., such that for every $j < J'$, $\pi(e_j)=2n_j$ and $\pi(o_j)=2n_j+1$; 
note that there exists $j \in J \setminus J'$ such that $2n_j \notin \dom(\pi)$ and $o(x)(2n_j)= a < e(z)(2n_j)=b < e(z)(2n_j+1)=c < o(x)(2n_j+1)=d$;
then $o(x) \sim f_\pi(o(x)) \prec e(z)$. 
\end{itemize}
The two points together gives $o(z) \prec e(z)$, i.e. $z \notin F$.
\vspace{2mm}

\item case $o(x) \prec e(x)$: this case is analog to the preceding one, where we just interchange the role of $o(x)$ and $e(x)$ in the construction, in particular we start with $x \notin F$ and we get $z \in F$ and in the construction we have to permute the elements in $O(l)$ instead of $E(l)$.

\vspace{3mm}

\item case $e(x) \sim o(x)$: Note $x \notin F$. Pick one triple $w_j$ and $n_j, n_j+1 \in w_j$ such that $j \in \odd$, and define 
\[
y(n) = \Big \{ 
\begin{array}{ll}
0 & \text{ if $(n=n_{j} \vee n=n_{j}+1)$}\\
x(n) & \text{ else} 
\end{array}
\]
By construction we obtain
\begin{itemize}
\item $\forall n \notin I_j :=\{2n_j,2n_j+1  \}, o(y)(n) = o(x)(n)$, otherwise one has $o(x)(2n_j)=a < o(y)(2n_j)=b < o(y)(2n_j+1)=c < o(x)(2n_j+1)=d$. Hence $o(x) \prec o(y)$.
\item $\forall n \notin I_j, e(y)(n) = e(x)(n)$, otherwise one has $e(y)(2n_j)=a < e(x)(2n_j)=b < e(x)(2n_j+1)=c < e(y)(2n_j+1)=d$. Hence $e(y) \prec e(x)$.
\end{itemize}
Hence $e(y) \prec e(x) \sim o(x) \prec o(y)$, i.e. $y \in F$.

\end{itemize}
\end{proof}

\begin{proposition}
Let $\preceq$ denote a SWR satisfying P and FA on $X= Y^{\omega}$, where $Y:=\{ 0,1  \}$.
Then there exists a social choice function $F$ which is not $D_\delta^+$-anti-democratic, for any $\delta \in (\frac{2}{3},1]$
\end{proposition}

\begin{proof}
The structure is very similar to the one above for SE and FA. The difference is that in this case we define $I_k:=[n_k,n_{k+1})$ and then 
\begin{equation}\label{E3}
o(x)(n):=  
\begin{cases}
0		&\text{ if $n \in [0,n_0)$} \\
1		&\text{ if $n \in E(x)$}\\
0		&\text{ if $n \in O(x)$}  \\
\end{cases}
\end{equation}
\begin{equation}\label{E4}
o(x)(n):=  
\begin{cases}
0		&\text{ if $n \in [0,n_0)$} \\
0		&\text{ if $n \in E(x)$}\\
1		&\text{ if $n \in O(x)$}  \\
\end{cases}
\end{equation}
The rest of the proof, splitting into the three cases, works following the same passages and we leave it to the reader. 
\end{proof}

\begin{remark}
It remains open what happens if one combines SE and P, without FA.
Assume there is a total SWR satisfying SE and P, can we then obtain a social choice function which is not $D^+_\delta$-anti-democratic?
\end{remark}

\section{A word about uncountable populations} \label{s7}

Assume $P=\kappa$, for $\kappa$ uncountable. 
Following the line of the previous section, we should find a version of Silver conditions in $2^\kappa$ with the property that the set of splitting levels be \emph{non-small}. 
We consider the \emph{club filter} $\club$ consisting of supersets of closed and unbounded subsets of $\kappa$.
Let $\mathsf{NS}$ be the non-stationary ideal, i.e. $\mathsf{NS}:= \{ x \in 2^\kappa: x^c \in \club  \}$. 
Then let $\clubsilver_\kappa:= \silver(\ns^*)$ consist of those conditions of the form
\[ 
N_f := \{ x \in 2^{\kappa}: \forall \alpha \in \dom(f) (x(\alpha)=f(\alpha)) \},
\]
with $\dom(f) \in \ns$.

In \cite[Lemma 4.1]{Lag15} it was proven that the club filer does not satisfy the $\clubsilver_\kappa$-property, i.e. for every $N_f \in \clubsilver_\kappa$
\[
N_f \cap \club \neq \emptyset \text{ and } N_f \not \subseteq \club.
\]

Note that following the interpretation of Section 2, this means that if we view $\club$ as a social choice function denoted by $F^\club$, then given any \emph{large} set $a$ of individuals, i.e. $a \in \club$, then $a$ is not $F^\club$-irrelevant. In fact we can always find a distribution of individuals choices in $a$ that generates a branch in $\club$ (i.e. for which the social choice function $F^\club$ takes value 1) and another distribution of individual choices in $a$ that generates a branch not in $\club$  (i.e. for which the social choice function $F^\club$ takes value 0). 

In terms of Definition \ref{democratic} we have then the following.
\begin{corollary} \label{corollary1}
$F^\club$ is not $\mathsf{NS}^*$-anti-democratic. 
\end{corollary}

Unfortunately, Corollary \ref{corollary1} cannot be improved to obtain social functions which are not $\mathsf{NS}^+$-anti-democratic. Indeed, if we consider $\silver(\ns^+)$ we then get a corresponding notion of $\silver(\ns^+)$-property that can be forced to hold for every set, when $\kappa$ is inaccessible. More precisely, in \cite[Lemma 4.2]{Lag15}, in case $\kappa$ be inaccessible, it is proven that  a $< \kappa$-support iteration of length $\kappa^+$ of $\kappa$-Cohen forcing provides a model-extension where all $\text{On}^\kappa$-definable subsets of $2^\kappa$ satisfy  the $\silver(\ns^+)$-property, and in the inner model  $L(\kappa^\kappa)$ of this forcing extension all sets satisfy the $\silver(\ns^+)$-property. In other words, the existence of a non-$\silver(\ns^+)$-set does not follow from the axiom of ZF without AC, and so it has a non-constructive nature.  

\emph{Question.} If $\kappa$ is not inaccessible, without using AC can we find an $F \subset 2^\kappa$ which is not $\mathsf{NS}^+$-anti-democratic?

\section{Concluding remarks} \label{s8}

The common tread linking the results unfolded in the paper is that principles excluding non-democratic behaviour of social choice functions (e.g. generalizing the Arrowian non-dictatorship) reveals a non-constructive nature independently of the other principles (such as Unanimity and Independence of Irrelevant Alternatives), and are therefore \emph{non-definable}. 

On the set-theoretic side, regularity properties of subsets of $2^{\N}$ are usually considered on the positive side, and on the contrary irregular sets are rather considered as pathological ones. In the framework suggested by social choice theory, this vision is reversed, as social choice functions satisfying notions like Silver-property and dense Silver-property, even if in different forms, tend to exhibit a non-democratic behaviour, since certain non-small coalitions are systematically irrelevant in the final collective choice. 

On the philosophical side, in the introduction we have dealt with the use of infinite populations in the context of economic and social choice theory, taking a position in line with Mihara's.  In this section we conclude with some insights  about the use of the axiom of choice AC and the axiom of determinacy AD. The position we want to keep is somehow in between, with  some caution in the use of AD. On the one side, we agree with Litak that a full and indiscriminate use of AC should be looked at with skepticism, and even rejected. On the other side, we think that a full use of AD can rise more than a suspect as well; indeed Corollary \ref{corollary1} suggests to also consider a certain limitation in using AD, as it essentially asserts that in every model of ZF+AD any social choice function exhibits a rather annoying non-democratic behaviour.

As a consequence, rather than considering AD as an axiom to select the appropriate applications where the use of infinite populations is legitimate, we tend to support the idea of considering as more suitable models for social choice theory on infinite populations also those satisfying certain fragments of AC excluding situation as in Corollary \ref{corollary1}.
As an example of an inappropriate model, one can consider $L(\mathbb{R})^{V[G]}$ in Lemma \ref{cohen-uppersilver}. In fact, in $L(\mathbb{R})^{V[G]}$ all social choice functions are $D_\delta^+$-anti-democratic and simultaneously there is a $\DDelta^1_2$-set non-Lebesgue measurable. Hence, $L(\mathbb{R})^{V[G]}$ might be considered unfit both from the point of view of social choice theory and from the point of view of economic theory as well, the latter since it violates Aumann's statement: ``\emph{Non-measurable sets are extremely ``pathological"; it is unlike that they would occur in the context of an economic model} \cite[pg. 44]{Au64}".

Under our consideration, a suitable model could be one where all sets are Lebesgue measurable (and so satisfying $\neg \nl$, ruling out Aumann's pathological situation) but simultaneously containing social choice functions which exclude anti-democratic features. For example, non-$\uppersilver$-sets could serve for the latter, but it is not clear whether one can have a model satisfying $\nuv \land \neg \nl$. In fact, it is possible that a result like $\nb \Rightarrow \nuv$ hold by replacing $\nb$ with $\nl$. In fact, it is not clear whether a result like in Lemma \ref{silver-cohen-2} can be proven when replacing comeager sets with sets having strictly positive measure.

\end{document}